\documentclass[english,10pt]{amsart}

\usepackage{amsmath, amssymb, amsthm, enumerate}
\usepackage[all]{xy}
\usepackage[activeacute]{babel}

\newtheorem{thm}[equation]{Theorem}
\makeatletter
\let\c@subsubsection\c@equation
\makeatother
\newtheorem{lem}[equation]{Lemma}
\newtheorem{prop}[equation]{Proposition}
\newtheorem{cor}[equation]{Corollary}

\newtheorem{conj}[equation]{Conjecture}

\theoremstyle{remark}
\newtheorem{rmk}[equation]{Remark}

\theoremstyle{definition}
\newtheorem{defi}[equation]{Definition}

\newcommand{\Hom}{\mathrm{Hom}}
\newcommand{\inthomeff}{\mathbf{hom}^{\mathrm{eff}}}
\newcommand{\Gm}{\mathbb G _{m}}

\newcommand{\spec}[1]{\mathrm{Spec}(#1)}
\newcommand{\generators}{\mathcal G}

\newcommand{\hr}{\sphere _{R}}

\newcommand{\sphere}{\mathbf 1}

\newcommand{\DMeff}{DM^{\mathrm{eff}}}

\newcommand{\northogonal}[1]{DM ^{\perp}(#1)}

\numberwithin{equation}{subsection}

\begin{document}


\title{Mixed Motives and Motivic Birational Covers}

\dedicatory{Dedicated to Professor Charles Weibel on the occasion of
	his 65th birthday.}


\author{Pablo Pelaez}
\address{Instituto de Matem\'aticas, Ciudad Universitaria, UNAM, DF 04510, M\'exico}
\email{pablo.pelaez@im.unam.mx}


\subjclass[2010]{Primary 14C25, 14C35, 14F42, 19E15; Secondary 18G55, 55P42}

\keywords{Bloch-Beilinson-Murre filtration, Chow Groups, Filtration on the Chow Groups,
	Filtration on Motivic Cohomology, Mixed Motives, Motivic Cohomology, 
	Triangulated Category of Motives}


\begin{abstract}
We introduce a tower of localizing subcategories in  Voevodsky's big (closed under infinite
coproducts) triangulated category of motives.  We show that the tower induces a
 finite filtration on the motivic cohomology groups of smooth schemes over
a perfect field.  With rational coefficients, this finite filtration satisfies several of the properties
of the still conjectural Bloch-Beilinson-Murre filtration.
\end{abstract}

\thanks{Research partially supported by DGAPA-UNAM grant IA100814.}

\maketitle

\section{Introduction}  \label{sec.introd}

The main goal of this paper is to present an alternative approach to the conjectural
Bloch-Beilinson-Murre filtration \cite{MR923131}, \cite{MR558224}, \cite{MR1225267}
in the context of Voevodsky's triangulated category of
motives $DM$.  Traditionally, the  Bloch-Beilinson-Murre filtration is 
understood as an outcome of the conjectural motivic $t$-structure
\cite[p. 20-22]{MR923131}, \cite[Conj. 4.8]{MR1265533}, \cite{MR2735752}.  
Due to the lack of progress, it seems to the author that it is worth it  to relax the
conditions and look instead for a tower where the truncation functors are triangulated.
The great advantage of this approach is that the finiteness of the proposed filtration
follows in a straightforward way from the construction, whereas in the traditional approach
this property seems to be the most inaccessible one \cite{MR1389964}, \cite{MR1744952}.
This method was introduced by Voevodsky in his successful approach to the spectral
sequence relating motivic cohomology and algebraic $K$-theory \cite{MR1977582},
\cite{MR1890744}.

Our approach can be sketched quickly as follows.  For a smooth scheme $X$ of finite type
over a perfect field $k$, the Chow groups can be computed in Voevodsky's triangulated
category of motives $DM$ \cite{MR1883180}: 
\[CH^{q}(X)_{R}\cong \Hom _{DM}(M(X)(-q)[-2q], \sphere _{R});\] 
where  $CH^{q}(X)_{R}$ is the Chow group with $R$-coefficients, and $\hr$ is the motive
of a point with $R$-coefficients.  
Since $DM$ is a triangulated category,
it is possible to construct the filtration by considering a tower in $DM$,  (see 
\S \ref{sec.biratHZ}):
\[	\cdots \rightarrow bc_{\leq -3}(\hr) \rightarrow bc_{\leq -2}(\hr) \rightarrow bc_{\leq -1}(\hr) 
	\rightarrow \hr \]
where $bc_{\leq -n}$ is a triangulated functor $DM \rightarrow DM$ \eqref{def.birat.cover},
and defining the $p$-component of the filtration $F^{p}CH^{q}(X)$ to be the image of the
induced map:
\[ \xymatrix@C=0.8pt{\Hom _{DM}(M(X)(-q)[-2q], 
	bc_{\leq -p}\hr) \ar[rr] &&
	 \Hom _{DM}(M(X)(-q)[-2q], \hr)\cong CH^{q}(X)_{R}}
\]
The finiteness of the filtration $F^{\bullet}CH^{q}(X)$ is proved in
\ref{thm.filt.motcoh.finite}.  Another advantage of this approach is that
the filtration is defined (and is finite) for any coefficient ring $R$ 
(not just the rationals) and for any smooth $k$-scheme of finite type (not necessarily
projective), whereas it is known that with integral coefficients it is not possible
to construct a motivic $t$-structure in $DM$ \cite[Prop. 4.3.8]{MR1764202}.

Now, we describe the contents of the paper.  In \S \ref{subsec.defandnots}
we fix the notation and introduce the basic definitions that will be used in the
rest of the paper.  
In \S \ref{sec.orth}, we review some general facts from triangulated categories that will
be used to construct the  birational towers.
In \S \ref{sec.birational}, we introduce the  birational covers and the  birational
tower which will
be used to construct the filtrations we are interested in, see \ref{eq.birtower},
\ref{def.birat.cover},
\ref{thm.oct-ax}, \ref{thm.birspecseq}.
In \S \ref{sec.KSunrc}, we compare the zero  birational cover constructed in this
paper with the Kahn-Sujatha unramified cohomology functor.
In \S \ref{sec.biratHZ}, we study the  birational
tower for motivic cohomology and describe the main properties of the induced filtration
on the motivic cohomology groups of smooth (projective) $k$-schemes,
see \ref{thm.filt.motcohY}, \ref{thm.filt.motcoh.finite}, \ref{prop.crit1}, 
\ref{thm.main2}.
Finally, in \S \ref{subsec.BBMfil} we include some  comments on the 
Bloch-Beilinson-Murre filtration.

\subsection{Definitions and Notation}	\label{subsec.defandnots}		

In this paper $k$ will denote a perfect base field, $Sch_{k}$ the category of $k$-schemes
of finite type and $Sm_{k}$ the full
subcategory of $Sch_{k}$ consisting of smooth $k$-schemes regarded
as a site with the Nisnevich topology.	  
	
Let $Cor_{k}$ denote the Suslin-Voevodsky category of finite correspondences over $k$,
having the same objects as $Sm_{k}$, morphisms $c(U,V)$ given by the group of finite
relative cycles on $U\times _{k}V$ over $U$ \cite{MR1764199}	 and composition as in
\cite[p. 673 diagram (2.1)]{MR2804268}.  By taking the graph of a morphism in $Sm_{k}$,
we obtain a functor $\Gamma : Sm_{k}\rightarrow Cor_{k}$.  A Nisnevich sheaf with
transfers is an additive contravariant functor $\mathcal F$ from $Cor_{k}$ to the category
of abelian groups such that the restriction $\mathcal F \circ \Gamma$ is a Nisnevich
sheaf.  We will write $Shv^{tr}$ for the category of Nisnevich sheaves with
transfers which is an abelian category \cite[13.1]{MR2242284}.  For $X\in Sm_{k}$,
let $\mathbb Z _{tr}(X)$ be the Nisnevich sheaf with transfers represented by $X$
\cite[2.8 and 6.2]{MR2242284}.

Let $K(Shv ^{tr})$ be the category of chain complexes (unbounded) on $Shv ^{tr}$ equipped
with the injective model structure \cite[Prop. 3.13]{MR1780498}, and let
$D(Shv ^{tr})$ be its homotopy category.  We will write $K^{\mathbb A ^{1}}(Shv ^{tr})$ for
the left Bousfield localization \cite[3.3]{MR1944041} of $K(Shv ^{tr})$ with respect to the set
of maps
$\{ \mathbb Z _{tr}(X\times _{k}\mathbb A ^{1})[n]\rightarrow \mathbb Z _{tr}(X)[n]:
X\in Sm_{k}; n\in \mathbb Z\}$ induced by the projections $p:X\times _{k}\mathbb A ^{1}
\rightarrow X$. Voevodsky's big triangulated category of effective motives
$\DMeff$ is the homotopy category of $K^{\mathbb A ^{1}}(Shv ^{tr})$ \cite{MR1764202}.

Let $T\in K^{\mathbb A ^{1}}(Shv ^{tr})$ be the chain complex of the from 
$\mathbb Z_{tr}(\Gm)[1]$ \cite[2.12]{MR2242284}, where $\Gm$ is the $k$-scheme
$\mathbb A ^1\backslash \{ 0\}$ pointed by $1$.  
We will write $Spt_{T}(Shv ^{tr})$ for the category of symmetric
$T$-spectra on $K^{\mathbb A ^{1}}(Shv ^{tr})$ equipped with the model structure
defined in \cite[8.7 and 8.11]{MR1860878}, \cite[Def. 4.3.29]{MR2438151}.  
Voevodsky's big triangulated category of  motives
$DM$ is the homotopy category of $Spt_{T}(Shv ^{tr})$ 
\cite{MR1764202}.  

We will write $M(X)$ for the image of $\mathbb Z _{tr}(X)\in D(Shv^{tr})$, $X\in Sm_{k}$
under the $\mathbb A ^{1}$-localization map $D(Shv ^{tr})\rightarrow \DMeff$.  Let
$\Sigma ^{\infty}:\DMeff \rightarrow DM$ be the suspension functor 
\cite[7.3]{MR1860878} (which is denoted by $F_0$ in \emph{loc.\!\,cit.}), 
we will abuse notation and simply write $E$ for $\Sigma
^{\infty}E$, $E\in \DMeff$.  Given a map $f:X\rightarrow Y$ in $Sm _k$, we will 
write $f:M(X)\rightarrow M(Y)$ for the map induced by $f$ in $DM$.

By construction, $\DMeff$ and $DM$ are tensor triangulated categories 
\cite[Thm. 4.3.76 and Prop. 4.3.77]{MR2438151} with unit $\sphere = M(\spec{k})$. 
We will write $E(1)$ for $E\otimes \mathbb Z _{tr}(\Gm)[-1]$, $E\in DM$ and inductively
$E(n)=(E(n-1))(1)$, $n\geq 0$.  We observe that the functor $DM\rightarrow DM$,
$E\mapsto E(1)$ is an equivalence of categories \cite[8.10]{MR1860878}, 
\cite[Thm. 4.3.38]{MR2438151};
we will write $E\mapsto E(-1)$ for
its inverse, and inductively $E(-n)=(E(-n+1))(-1)$,
$n>0$.  By convention $E(0)=E$ for $E\in DM$.

Let $R$ be a commutative ring with $1$.  We will write  
$E_{R}$ for $E\otimes \sphere _{R}$ where $E\in DM$, and $\sphere _R$
is the motive of a point with $R$-coefficients $M(\spec{k})\otimes R$.
	
We shall use freely the language of triangulated categories.  Our main reference will 
be \cite{MR1812507}.  Given a triangulated category, we will write $[1]$ 
(resp. $[-1]$) to denote its suspension 
(resp. desuspension) functor; and for $n>0$, $[n]$ (resp. $[-n]$)
will be the composition of $[1]$
(resp. $[-1]$) iterated $n$-times.  If $n=0$,
$[0]$ will be the identity functor.

We will use the following notation in all the categories under consideration: $0$ will
denote the zero object, and $\cong$ will denote that a map (resp. a functor) is an
isomorphism (resp. an equivalence of categories).

\section{Localizing and orthogonal subcategories}  \label{sec.orth}
	
In this section we collect several results of Neeman which are the main technical tools
for the construction of the filtration on motivic cohomology.  All the constructions 
can be carried out
as well at the level of model categories, see \cite[\S 2]{Pelaez:2014}.

\subsection{} \label{subsec.somefactsSH}

Let $\mathcal T$ be a compactly generated triangulated category in the sense of Neeman
\cite[Def. 1.7]{MR1308405} with set of compact generators  $\mathcal G$.  For 
$\mathcal G '\subseteq \mathcal G$, let $Loc(\mathcal G ')$ denote the smallest full
triangulated subcategory of $\mathcal T$ which contains $\mathcal G'$ and is closed
under arbitrary (infinite) coproducts. 

\begin{defi}  \label{def.orthn}
Let $\mathcal T '\subseteq \mathcal T$  be a triangulated 
subcategory.		We will write $\mathcal T '^{\perp}$ for the full subcategory of
$\mathcal T$ consisting of the objects $E\in \mathcal T$ such that
for every $K\in \mathcal T '$:  $\Hom _{\mathcal T}(K,E)=0$.

If $\mathcal T '=Loc(\mathcal G')$ and $E\in \mathcal T '^{\perp}$,
we will say that $E$ is \emph{$\mathcal G'$-orthogonal}.
\end{defi} 

\begin{rmk}  \label{rmk.orthchk}
In order to check that $E\in Loc(\mathcal G')^{\perp}$ it suffices to see that for every
$G\in \mathcal G '$ and for every $p\in \mathbb Z$: $\Hom _{\mathcal T}(G[p],E)=0$.  
In effect, let $^{\perp}E$ be the full
subcategory of $\mathcal T$ consisting of objects $F$ such that $\Hom_{\mathcal T}
(F[p],E)=0$ for every $p\in \mathbb Z$.  It is clear that $^{\perp}E$ is closed under
infinite coproducts and that it is a full triangulated subcategory of $\mathcal T$.  Now,
our hypothesis implies that $\mathcal G' \subseteq ^{\perp}\!\! E$.  Hence,
$Loc(\mathcal G')\subset ^{\perp}\!\! E$ by definition of $Loc(\mathcal G ')$
\eqref{subsec.somefactsSH}, and this is equivalent to $E\in Loc(\mathcal G')^{\perp}$
\eqref{def.orthn}.
\end{rmk}

\begin{thm}[Neeman]  \label{thm.gentower}
Under the same hypothesis as in  \ref{subsec.somefactsSH}.  Let $\mathcal T'$
be a triangulated subcategory of $\mathcal T$.  Assume that $\mathcal T'$
is also compactly generated.
Then the inclusion $i:\mathcal T ' \rightarrow \mathcal T$ admits a right adjoint $r$,
which is a triangulated functor. 

Let $f=i\circ r$.  Then
there exists a triangulated functor $s:\mathcal T \rightarrow \mathcal T$
	together with natural transformations:
		\[ \xymatrix@R=.6pt{\pi :id \ar[r]& s \\
								\sigma :s \ar[r]& [1]\circ f }
		\]
	such that for any  $E$ in $\mathcal T$ the following conditions hold:
	\begin{enumerate}
		\item	\label{thm.gentower.a} There is a natural
					distinguished triangle in $\mathcal T$:
					\begin{equation}
						\label{eq.gentower}
							\xymatrix{f(E) \ar[r]& E \ar[r]^-{\pi}& 
							 s(E) \ar[r]^-{\sigma}& f(E)[1]}
					\end{equation}
		\item	\label{thm.gentower.b} $s(E)$ is in
					$\mathcal T '^{\perp}$ (see \ref{def.orthn}). 
	\end{enumerate}
\end{thm}
\begin{proof}
The triangulated categories $\mathcal T$ and $\mathcal T '$ are compactly generated.
Hence, the existence of the right adjoint $r$ follows from theorem 4.1 in \cite{MR1308405}.  
The remaining results follow from propositions 9.1.19 and 9.1.8 in \cite{MR1812507}.	
\end{proof}

\subsubsection{}  \label{neemanleftadj}
With the notation of \ref{subsec.somefactsSH} and \ref{thm.gentower}.  If
$\mathcal T' =Loc (\mathcal G')$, we will write $f_{\mathcal G'}$, $s_{\mathcal G '}$,
$\pi _{\mathcal G '}$, $\sigma _{\mathcal G '}$ for $f$, $s$,
$\pi$, $\sigma$ respectively.

For $E\in \mathcal T$, consider the natural distinguished triangle
\eqref{eq.gentower}.  Thus, if $E'\in Loc(\mathcal G ')^{\perp}$ we 
deduce that:
\begin{equation}  \label{eqn.adjorth}
	\Hom_{\mathcal T}(E,E')\cong \Hom_{\mathcal T}(s_{\mathcal G '}E,E')
	\cong \Hom_{Loc(\mathcal G ')^{\perp}}(s_{\mathcal G '}E,E').
\end{equation}

\begin{lem} \label{lem.orthpr}
Under the same hypothesis as in  \ref{subsec.somefactsSH} and with
the notation of \ref{neemanleftadj}.
Consider the inclusion $j: Loc(\mathcal G')^{\perp} \rightarrow \mathcal T$.
\begin{enumerate}
\item \label{lem.ortcp.a} $Loc(\mathcal G')^{\perp}$ is closed under infinite coproducts 
and it is a full triangulated subcategory of $\mathcal T$.
\item \label{lem.ortcp.b} The inclusion 
$j$  is a triangulated functor
which commutes with infinite coproducts.  In addition, $Loc(\mathcal G')^{\perp}$ is a
compactly generated triangulated category in the sense of Neeman 
\cite[Def. 1.7]{MR1308405}
with set of compact generators $\mathcal G  ^{\prime \perp}=
\{ s_{\mathcal G '}G: G\in \generators 
\backslash \generators ' \}$. 
\item \label{lem.ortcp.c}The inclusion $j$
admits a right adjoint $p:\mathcal T \rightarrow Loc(\mathcal G ') ^{\perp}$,
which is also a triangulated functor.
\end{enumerate}
\end{lem}
\begin{proof}
\eqref{lem.ortcp.a}:  It is clear from the definition \eqref{def.orthn} that 
$Loc(\mathcal G')^{\perp}$ is a full triangulated subcategory of $\mathcal T$.  

Now we proceed to show that
$Loc(\mathcal G')^{\perp}$ is closed under infinite coproducts.  Consider an indexing set 
$\Lambda$
and let $E=\oplus_{\lambda \in \Lambda}E_{\lambda}$ where $E_{\lambda}\in
Loc(\mathcal G')^{\perp}$ for all $\lambda$. 
By \ref{rmk.orthchk} it suffices to see that for all $G\in \mathcal G '$: 
$\Hom_{\mathcal T}(G[p],E)=0$ for every $p\in \mathbb Z$.
However, $G$ is a compact object in $\mathcal T$; thus we conclude that 
$\Hom_{\mathcal T}(G[p],E)\cong \oplus _{\lambda \in \Lambda}
\Hom_{\mathcal T}(G[p],E_{\lambda})=0$ since $G[p]\in Loc(\mathcal G')$
and $E_{\lambda}\in Loc(\mathcal G')^{\perp}$.

\eqref{lem.ortcp.b}:
By \ref{lem.orthpr}\eqref{lem.ortcp.a}, $j$ is a triangulated functor that commutes with infinite
coproducts.  Now, let $G\in \generators ^{\prime \perp}\subseteq \generators$.  
Since $G$ is compact \cite[Def. 1.6]{MR1308405} in 
$\mathcal T$, it follows from \eqref{eqn.adjorth} that $s_{\mathcal G '}G$ is also 
compact in $Loc(\mathcal G')^{\perp}$.

Thus, it only remains to see that $\generators ^{\prime \perp}$ is a set of
generators for $Loc(\mathcal G')^{\perp}$ \cite[Defs. 1.7, 1.8]{MR1308405}.  To check this,
let $E\in Loc(\mathcal G')^{\perp}$ be such that $\Hom _{Loc(\mathcal G')^{\perp}}
(s_{\generators '}G,E)\cong \Hom _{\mathcal T}(G,E)=0$ \eqref{eqn.adjorth}
for all $G\in \generators \backslash \generators '$. We observe that if $G \in \mathcal G '$
then $\Hom _{\mathcal T} (G,E)=0$ since 
$E\in Loc(\mathcal G')^{\perp}$ (see \ref{def.orthn}).  Therefore, $\Hom _{\mathcal T}
(G,E)=0$ for all $G\in \generators =\mathcal G ' \cup (\generators \backslash \generators ')$,
and we conclude that $G\cong 0$ since
$\generators$ is a set of generators for $\mathcal T$.

\eqref{lem.ortcp.c}: This follows by combining \ref{lem.orthpr}\eqref{lem.ortcp.b} and
\ref{thm.gentower}.
\end{proof}

\subsubsection{} \label{rmk.unit=iso}
Since the inclusion $j: Loc(\mathcal G ')^{\perp} \rightarrow \mathcal T$
is a full embedding,
we deduce that the unit of the adjunction $id\stackrel{\tau}{\rightarrow} pj$
is a natural isomorphism.

\subsection{Slice towers and their duals}  \label{subsec.slcos}
Recall that $\mathcal T$ is a compactly generated triangulated category with set of
compact generators $\generators$.

\subsubsection{}  \label{subsubsec.gentow}
Consider a family of subsets of $\generators$: $\mathcal S=\{ \generators _{n} \} _{n
\in \mathbb Z}$
such that $\generators _{n+1}\subseteq \generators _{n}\subseteq \generators$ for
every $n\in \mathbb Z$.

Thus, we obtain a tower of full triangulated subcategories of $\mathcal T$:
\begin{equation}  \label{eq.genslitow}
\cdots \subseteq Loc(\generators _{n+1}) \subseteq Loc(\generators _{n})
\subseteq Loc(\generators _{n-1}) \subseteq \cdots
\end{equation}

We will call \eqref{eq.genslitow} the \emph{slice tower} determined by $\mathcal S$.  
The reason
for this terminology is \cite{MR1977582}, \cite{MR2249535}, \cite[p. 18]{MR2600283}.  
If we consider the orthogonal categories
$Loc(\generators _{n})^\perp$ \eqref{def.orthn}, we obtain a tower of full triangulated
subcategories of $\mathcal T$:
\begin{equation}  \label{eq.genslitow2}
\cdots \subseteq Loc(\generators _{n-1})^{\perp} \subseteq Loc(\generators _{n})^{\perp}
\subseteq Loc(\generators _{n+1})^{\perp} \subseteq \cdots
\end{equation}

\section{Birational Coverings and the Birational Tower}  
\label{sec.birational}

In this section we apply the formalism of \S \ref{sec.orth} to Voevodsky's triangulated
category of motives $DM$, in order to construct a tower of triangulated subcategories
of $DM$ which will induce the filtration on the Chow groups that we are interested in.
In \cite{Pelaez:2014} the construction was done in the Morel-Voevodsky
motivic stable homotopy category $\mathcal{SH}$.  However, as pointed out by the
anonymous referee, it is more natural to carry out the construction in $DM$ since
various issues of functoriality with respect to Chow correspondences become
straightforward in $DM$.  Nevertheless the analogue tower in $\mathcal{SH}$ is also
interesting, since it is possible to show that with finite coefficients the algebraic cycles induced by the
Steenrod operations of Voevodsky are elements of high order in the filtration, and it 
allows to study other theories which are not representable in $DM$
(integrally), e.g. homotopy invariant $K$-theory $KH$ and Voevodsky algebraic cobordism
$MGL$.  We will study this applications for $\mathcal{SH}$ in a future work.

\subsection{Generators}  \label{sub.gens}

It is well known that
$DM$ is a compactly generated triangulated category \eqref{subsec.somefactsSH} with
compact generators \cite[Thm. 4.5.67]{MR2438151}:
\begin{equation}  \label{eq.DMgens}
 \generators _{DM}=\{ M(X)(p): X\in Sm_{k}; p \in \mathbb Z\}.
\end{equation}

Let $\generators ^{\mathrm{eff}}\subseteq \generators _{DM}$ be the set consisting of compact objects of the form:
\begin{equation}  \label{eq.DMeffgens}
 \generators ^{\mathrm{eff}}=\{ M(X)(p): X\in Sm_{k}; p\geq 0\}.
\end{equation}

If $n\in \mathbb Z$, we will write $\generators ^{\mathrm{eff}}(n)\subseteq 
\generators _{DM}$ for the set consisting of compact objects of the form:
\begin{equation}  \label{eq.DMeffgenstw}
 \generators ^{\mathrm{eff}}(n)=\{ M(X)(p): X\in Sm_{k}; p\geq n\}.
\end{equation}

\subsubsection{}  \label{subsubsec.cancel}
By Voevodsky's cancellation theorem \cite{MR2804268}, the suspension functor 
$\Sigma ^{\infty}:\DMeff \rightarrow DM$ induces an equivalence of categories between 
$\DMeff$ and the full triangulated subcategory $Loc(\generators ^{\mathrm{eff}})$ of $DM$
\eqref{subsec.somefactsSH}.
We will abuse notation and write $\DMeff$ for $Loc(\generators ^{\mathrm{eff}})$.

\subsubsection{}  \label{subsubsec.neffDM}
We will write $\DMeff (n)$ for the full triangulated subcategory $Loc(\generators
^{\mathrm{eff}} (n))$ of $DM$ \eqref{subsec.somefactsSH}, and $\northogonal{n}$
for the orthogonal category $Loc(\generators^{\mathrm{eff}} (n))^{\perp}$
\eqref{def.orthn}.  Notice that $\DMeff (n)$ is compactly generated with set of generators
$\generators ^{\mathrm{eff}}(n)$ \cite[Thm. 2.1(2.1.1)]{MR1308405}.

\subsection{The  birational tower}  \label{subsec.birattow}  

Consider the family $\mathcal S=\{ \generators ^{\mathrm{eff}}(n) \}_{n\in \mathbb Z}$
of subsets of $\generators _{DM}$.  By construction $\mathcal S$ satisfies the conditions of
\ref{subsubsec.gentow},  hence we obtain a tower of full triangulated subcategories of
$DM$ \eqref{eq.genslitow2}:
		\begin{align}  \label{eq.birtower}
			\cdots \subseteq \northogonal{q-1}
			\subseteq \northogonal{q}
			\subseteq \northogonal{q+1} \subseteq \cdots
		\end{align}	
We will call the tower \eqref{eq.birtower} the \emph{birational tower}.  The reason
for this terminology is \cite{Kahn:2015qf}, \cite[Thms. 3.6 and 1.4.(2)]{MR3035769} where it is shown that the orthogonal categories
have a geometric description in terms of birational conditions.  This also explains the shift in the index in \ref{def.birat.cover}.

\begin{prop}  \label{prop.adj-ort.b}
	The inclusion,
		$ j_{q}: \northogonal{q} \rightarrow DM$
	admits a right adjoint:
		\[p_{q}:DM \rightarrow \northogonal{q},\]
	which is also a triangulated functor.
\end{prop}
\begin{proof}
Since $DM$ is compactly generated, we can apply 
\ref{lem.orthpr}\eqref{lem.ortcp.c}.
\end{proof}
	
\subsubsection{Birational covers} \label{def.birat.cover}		
		
We define $bc_{\leq q}=j_{q+1}\circ p_{q+1}$.  
	
The following proposition is well-known.
\begin{prop}  \label{counit-properties}
	The counit $bc_{\leq q}=j_{q}p_{q}\stackrel{\theta_{q}}{\rightarrow} id$ of the adjunction
	 constructed in \ref{prop.adj-ort.b}
	satisfies the following universal property: 
	
	For any $E$ in $DM$ and for any 
	$F\in \northogonal{q+1}$,
	the map $\theta ^{E}_{q}: bc_{\leq q}E \rightarrow E$
	in $DM$ induces an isomorphism of
	abelian groups:
		\[\xymatrix{\Hom _{DM}(F, bc_{\leq q}E) \ar[r]_-{\cong}^-{\theta ^{E}_{q\ast}}
		& 
			\Hom _{DM}(F, E) }
		\]
\end{prop}
\begin{proof}
	If $F\in \northogonal{q+1}$, then $\Hom _{DM}(F, E)=\Hom
	 _{DM}(j_{q+1}F, E)$.
	By adjointness:
		\[  \Hom _{DM}(j_{q+1}F, E)\cong \Hom _{\northogonal{q+1}}(F,p_{q+1}E).
		\]
	Since $\northogonal{q+1}$ is a full subcategory of $DM$, we deduce that:
	\[ \Hom _{\northogonal{q+1}}(F,p_{q+1}E)=
	\Hom _{DM}(j_{q+1}F, j_{q+1}p_{q+1}E)=
	\Hom _{DM}(F, bc_{\leq q}E).
	\]  
	This finishes the proof.
\end{proof}

\begin{rmk}  \label{rmk.univcar}
By construction $bc_{\leq q}E$ is in $\northogonal{q+1}$ 
\eqref{prop.adj-ort.b}-\eqref{def.birat.cover}, thus we conclude that the universal
property of \ref{counit-properties} characterizes $bc_{\leq q}E$ up to a unique
isomorphism.
\end{rmk}

\subsubsection{}  \label{subsec.bc.idemp}
Since $\northogonal{q+1}\subseteq \northogonal{q+2}$ \eqref{eq.birtower}, 
it follows from \eqref{counit-properties}-\eqref{rmk.univcar} that
$bc_{\leq q} \circ bc_{\leq q+1}\cong bc_{\leq q}$  and that there
exists  a canonical natural transformation $bc_{\leq q}
\rightarrow bc_{\leq q+1}$.

\begin{cor}  \label{cor.comp.tool1}
	Let $E$ be in $DM$.  Then, the natural
	 map $\theta ^{E}_{q}: bc_{\leq q}E\rightarrow E$
	is an isomorphism in $DM$ if and only if $E$ belongs to $\northogonal{q+1}$.
\end{cor}
\begin{proof}
	First, we assume that $\theta ^{E}_{q}$ is an isomorphism in $DM$.  We observe that
	$bc_{\leq q}E$ is in $\northogonal{q+1}$ (see \ref{prop.adj-ort.b} and \ref{def.birat.cover}).   
	Hence, we deduce that $E$ is also in $\northogonal{q+1}$ since
	it is a full triangulated subcategory of $DM$.
	
	Finally, we assume that $E$ is in $\northogonal{q+1}$.  Thus,
	$\theta ^{E}_{q}$ is a map in $\northogonal{q+1}$ since $bc_{\leq q}E$ is in
	$\northogonal{q+1}$ by construction.  Therefore, by \ref{counit-properties} we
	 deduce that
	$\theta ^{E}_{q}$ is an isomorphism in $\northogonal{q+1}$, and hence an isomorphism in
	 $DM$.
\end{proof}

\begin{thm}  \label{dual.slicefil}
There exist triangulated functors: 
\[ bc_{q+1/q}:DM \rightarrow DM
\]
together with natural transformations:
		\[ \xymatrix@R=.6pt{\pi _{q+1}: bc_{\leq q+1} \ar[r]& bc_{q+1/q} \\
								\sigma _{q+1}: bc_{q+1/q} \ar[r]& [1]\circ bc_{\leq q}}
		\]
	such that for any $E$ in $DM$ the following
	 conditions hold:
	\begin{enumerate}
		\item	\label{dual.slicefil.a} There is a natural
					distinguished triangle in $DM$:
					\begin{equation}
						\label{eq.dual.slicefil}
							\xymatrix{bc_{\leq q}E \ar[r]& bc_{\leq q+1}E \ar[r]^-{\pi_{q+1}}& 
							 bc_{q+1/q}E \ar[r]^-{\sigma_{q+1}}& bc_{\leq q}E[1]}
					\end{equation}
		\item	\label{dual.slicefil.b} $bc_{q+1/q}E$ is in $\northogonal{q+2}$.
		\item	\label{dual.slicefil.c} $bc_{q+1/q}E$ is in $(\northogonal{q+1})^{\perp}$.
		  Namely,
					for any  $F$ in $\northogonal{q+1}$:
					\[ \Hom _{DM}(F, bc_{q+1/q}E)=0.\]  
	\end{enumerate}
\end{thm}
\begin{proof}
Since $DM$ is compactly generated, we can apply \ref{lem.orthpr}\eqref{lem.ortcp.b}
to conclude that the triangulated categories $\northogonal{q+1}\subseteq
\northogonal{q+2}$ (see \ref{eq.birtower})
are compactly generated.  Thus,  the result follows by combining
\ref{thm.gentower} (see \ref{prop.adj-ort.b}-\ref{def.birat.cover}) and \ref{subsec.bc.idemp}.
\end{proof}

\begin{thm}  \label{thm.dbdual}
There exist triangulated functors:
\[ bc_{>q}:DM \rightarrow DM
\]
	together with natural transformations:
		\[ \xymatrix@R=.6pt{\pi _{>q}:id \ar[r]& bc_{>q} \\
								\sigma _{>q}:bc_{>q} \ar[r]& [1]\circ bc_{\leq q}}
		\]
	such that for any $E$ in $DM$ the following
	 conditions hold:
	\begin{enumerate}
		\item	\label{thm.dbdual.a} There is a natural
					distinguished triangle in $DM$:
					\begin{equation}
						\label{eq.dbdual}
							\xymatrix{bc_{\leq q}E \ar[r]& E \ar[r]^-{\pi_{>q}}& 
							 bc_{>q}E \ar[r]^-{\sigma_{>q}}& (bc_{\leq q}E)[1]}
					\end{equation}
		\item	\label{thm.dbdual.b} $bc_{>q}E$ is in
					$(\northogonal{q+1})^{\perp}$.  Namely,  for any $F$ in $\northogonal{q+1}$:
					\[\Hom _{DM}(F,bc_{>q}E)=0.\]  
	\end{enumerate}
\end{thm}
\begin{proof}
Combining \ref{sub.gens} and \ref{lem.orthpr}\eqref{lem.ortcp.b}, we deduce that
the triangulated categories $\northogonal{q+1}\subseteq DM$ 
are compactly generated.  Thus,  the result follows from 
\ref{thm.gentower} (see \ref{prop.adj-ort.b}-\ref{def.birat.cover}).	
\end{proof}

\begin{thm}
		\label{thm.oct-ax}
For any $E$ in $DM$, there exists the following commutative diagram in $DM$: 
\[ \xymatrix{bc_{\leq q}E \ar[rr] \ar@{=}[d] && bc_{\leq q+1}E \ar[rr]^-{\pi _{q+1}} \ar[d]&&
	bc_{q+1/q}E \ar[rr]^-{\sigma _{q+1}} \ar[d]&& (bc_{\leq q}E)[1] \ar@{=}[d]\\
	bc_{\leq q}E \ar[rr] \ar[d]&& E \ar[rr]^-{\pi _{>q}} \ar[d]_-{\pi _{>q+1}}&&
	bc_{>q}E \ar[rr]^-{\sigma _{>q}} \ar[d]&& (bc_{\leq q}E)[1] \ar[d]\\
	0 \ar[rr] \ar[d]&& bc_{>q+1}E \ar@{=}[rr] \ar[d]_-{\sigma _{>q+1}}&& 
	bc_{>q+1}E \ar[rr] \ar[d]&& 0 \ar[d]\\
	(bc_{\leq q}E)[1] \ar[rr]&& (bc_{\leq q+1}E)[1] \ar[rr]_-{[1]\circ \pi_{q+1}}&&
	(bc_{q+1/q}E)[1] \ar[rr]_-{[1]\circ \sigma _{q+1}}&& (bc_{\leq q}E)[2]}
\]
where all the rows and columns are distinguished triangles in $DM$.
\end{thm}
\begin{proof}
	The result follows from  \ref{dual.slicefil}, \ref{thm.dbdual} and
	the octahedral axiom applied to the following commutative diagram
	(see \ref{subsec.bc.idemp}):
		\[  \xymatrix{bc_{\leq q}E \ar[rr] \ar[dr]&& bc_{\leq q+1}E \ar[dl]\\
								& E &}
		\]
\end{proof}

\subsubsection{The spectral sequence}

By \ref{thm.oct-ax}, for every $E$ in $DM$ there is a tower in $DM$:

\begin{align}	\label{F.birtow}
	\xymatrix@C=1.5pc{\cdots \ar[r]
					& \ar[dr]|{\theta _{-1}^{E}} bc_{\leq -1}(E) 
					\ar[r] & \ar[d]|{\theta _{0}^{E}} bc_{\leq 0}(E) \ar[r] 
					& bc_{\leq 1}(E) \ar[r] \ar[dl]|{\theta _{1}^{E}} & \cdots  \\						
					&&  E &&}
\end{align}
We will call \eqref{F.birtow} the birational tower of $E$.

\begin{rmk}  
	By \ref{prop.adj-ort.b} and \ref{def.birat.cover}, the birational tower \eqref{F.birtow}
	 is functorial with respect to
	morphisms in $DM$.
\end{rmk}

\begin{thm}  \label{thm.birspecseq}
	Let $G, K$ be in $DM$.  Then
	there is a spectral sequence of homological type with term $E^{1}_{p,q}=\Hom
	 _{DM}(G, (bc_{p/p-1}K)[q-p])$
	and where the abutment is given by the associated graded group for the filtration
	 $F_{\bullet}$ of $\Hom _{DM} (G,K)$ defined
	by the image of $\theta _{p\ast}^{K}:\Hom _{DM}(G, bc_{\leq p}K)\rightarrow
	 \Hom _{DM}(G, K)$, or
	equivalently the kernel of $\pi_{>q}:\Hom _{DM}(G, K)\rightarrow \Hom
	 _{DM}(G, bc_{>q}K)$.
\end{thm}
\begin{proof}
Since $DM$ is a triangulated category, the result follows from 
\ref{thm.oct-ax} and \ref{F.birtow}.
\end{proof}	

\subsection{Effective covers}  \label{subsec.ssfil}

Consider again the family $\mathcal S=\{ \generators ^{\mathrm{eff}}(n) \}_{n\in \mathbb Z}$
of subsets of $\generators _{DM}$ \eqref{eq.DMgens}-\eqref{eq.DMeffgenstw}.  
Since $DM$ and $\DMeff (n)$ are compactly generated (see \ref{sub.gens} and
\ref{subsubsec.neffDM}), it follows from \ref{thm.gentower} that the inclusion $i_{n}:\DMeff
(n)\rightarrow DM$ admits a right adjoint $r_{n}:DM\rightarrow \DMeff (n)$ which is a
triangulated functor.  We will write $f_{n}$ for the triangulated functor $i_{n}\circ r_{n}:DM
\rightarrow DM$.  

The functor $f_{n}$ is the $(n-1)$-effective cover in the slice filtration
studied in \cite{MR1977582}, \cite{MR2249535} in the context of the Morel-Voevodsky
motivic stable homotopy category $\mathcal{SH}$ and Voevodsky's
effective triangulated category of motives $\DMeff$, respectively.

\begin{rmk}  \label{rmk.univeffcar}
Since $f_{q}E\in \DMeff (q)$,
an argument parallel to \eqref{counit-properties}-\eqref{rmk.univcar} shows that
$f_{q}E$ is characterized up to a unique isomorphism by the following
universal property:

For any $F\in \DMeff (q)$, the counit of the adjunction $(i_{n},r_{n})$,
$\epsilon ^{E}_{q}: f_{q}E \rightarrow E$ in $DM$ induces an isomorphism of
abelian groups:
\[\xymatrix{\Hom _{DM}(F, f_{q}E) \ar[r]_-{\cong}^-{\epsilon ^{E}_{q\ast}} & 
	\Hom _{DM}(F, E) }
\]
\end{rmk}

\begin{lem}  \label{lem.comp.orteff}
Let $p$, $q\in \mathbb Z$ with $q\geq p$.  Then:
\begin{enumerate}
\item \label{lem.comp.orteff.a}  $f_{p}(bc_{\leq q}E)$ is in 
$\DMeff (p) \cap \northogonal{q+1}$.
\item \label{lem.comp.orteff.b}  $f_{p}(bc_{\leq q}E)$ is characterized up to a unique
isomorphism by the following universal property:

For any $F\in \DMeff (p)\cap \northogonal{q+1}$, the compositon of counits
(see \ref{counit-properties} and \ref{rmk.univeffcar})
$\theta ^{E}_{q}\circ \epsilon ^{bc_{\leq q}E}_{p}: f_{p}(bc_{\leq q}E) \rightarrow E$ in $DM$ 
induces an isomorphism of abelian groups:
\[\xymatrix{\Hom _{DM}(F, f_{p}(bc_{\leq q}E)) \ar[r]_-{\cong} & 
	\Hom _{DM}(F, E) }
\]
\end{enumerate}
\end{lem}
\begin{proof}
\eqref{lem.comp.orteff.a}:  We observe that $f_{p}(bc_{\leq q}E)$ is in $\DMeff (p)$
by construction.  Now, let $F\in \DMeff (q+1)$.  To conclude it suffices to check that
$\Hom _{DM}(F,f_{p}(bc_{\leq q}E))=0$.

We observe that $\DMeff (q+1)
\subseteq \DMeff (p)$ since $q+1>p$ (see \ref{eq.DMeffgenstw} and
\ref{subsubsec.neffDM}), so $F$ is also in $\DMeff (p)$.  Then, by the universal
property \ref{rmk.univeffcar}:
\[ \Hom _{DM}(F, f_{p}(bc_{\leq q}E)) \cong  
	\Hom _{DM}(F, bc_{\leq q}E).
\]
Finally, $\Hom _{DM}(F, bc_{\leq q}E)=0$ since $F\in \DMeff (q+1)$ and
$bc_{\leq q}E\in \northogonal{q+1}$ (see \ref{rmk.univcar}).

\eqref{lem.comp.orteff.b}:  Follows directly by combining \ref{rmk.univcar},
\ref{rmk.univeffcar} and \eqref{lem.comp.orteff.a} above.
\end{proof}

\begin{prop}  \label{prop.Tatetwist}
Let $E\in DM$ and $n$, $q\in \mathbb Z$.  Then:
\begin{enumerate}
\item \label{prop.Tatetwist.a}  There is a natural isomorphism
$t_n^{bc}(E):(bc_{\leq q}E)(n)\rightarrow bc_{\leq q+n}(E(n))$  such that
$\theta _{q+n} ^{E(n)} \circ t_n^{bc}(E)=\theta _q ^E (n)$ (see \ref{counit-properties}).
\item \label{prop.Tatetwist.b}  There is a natural isomorphism
$t_n^{\mathrm{eff}}(E):(f_{q}E)(n)\rightarrow f_{q+n}(E(n))$ such that
$\epsilon _{q+n} ^{E(n)} \circ t_n^{\mathrm{eff}}(E)=\epsilon _q ^E (n)$ (see
\ref{rmk.univeffcar}).
\end{enumerate}
\end{prop}
\begin{proof}
We observe that $DM\rightarrow DM$, $E\mapsto E(n)$ is a triangulated equivalence
of categories which maps $\DMeff (q)$ surjectively onto $\DMeff (q+n)$ (see
\ref{subsubsec.neffDM} and \ref{eq.DMeffgenstw}), and hence also
$\northogonal{q+1}$ surjectively onto $\northogonal{q+n+1}$.  Hence,
the result follows from \ref{counit-properties}-\ref{rmk.univcar} (resp.
\ref{rmk.univeffcar}).
\end{proof}

\section{Kahn-Sujatha unramified cohomology}  \label{sec.KSunrc}

In this section, we will show that the right adjoint $p_{1}:DM\rightarrow \northogonal{1}$
constructed in \ref{prop.adj-ort.b}
is a non-effective version of the Kahn-Sujatha unramified cohomology functor
$R_{nr}$ defined in \cite[\S 5]{Kahn:2015qf}.  The author would like to thank the
anonymous referee for bringing this crucial fact to our attention.

\subsection{Kahn-Sujatha birational motives}  \label{subsec.KSbirmot}

Let $DM^{\circ}$ be the Kahn-Sujatha triangulated category of birational motives defined in
\cite[Def. 3.2.1]{Kahn:2015qf}, $\nu_{\leq 0}:\DMeff \rightarrow DM^{\circ}$ the
corresponding localization functor (see the diagram in \cite[p. 20]{Kahn:2015qf}), and
$i^{\circ}:DM^{\circ} \rightarrow \DMeff$ its right adjoint \cite[Thm. 3.3.5]{Kahn:2015qf}.

\begin{rmk}  \label{rmk.counitiso}
Since the right adjoint $i^{\circ}$ is fully faithful 
\cite[Thm. 3.3.5]{Kahn:2015qf}, we conclude that the counit $\nu_{\leq 0}\circ i^{\circ}
\rightarrow id$ of the adjunction $(\nu_{\leq 0}, i^{\circ})$ is a natural isomorphism.
\end{rmk}

\subsubsection{}  \label{subsubsec.KSnot} 

We will write $\DMeff \cap \northogonal{1}$ for the full triangulated subcategory of $\DMeff$
consisting of objects $E$ which belong to $\DMeff$ and to $\northogonal{1}$
\eqref{subsubsec.neffDM}.  Let $\iota$ denote the inclusion $\DMeff \cap \northogonal{1}
\rightarrow \DMeff$.

\begin{lem}  \label{lem.KScomp1}
With the notation of \ref{subsec.KSbirmot} and \ref{subsubsec.KSnot}.
\begin{enumerate}
\item \label{lem.KScomp1.a}  The composition $\nu_{\leq 0} \circ \iota : \DMeff \cap
\northogonal{1}\rightarrow DM^{\circ}$ is an equivalence of categories.  

\item \label{lem.KScomp1.b}  For every $E\in DM^{\circ}$, $i^{\circ}E \in 
\DMeff \cap \northogonal{1}$.  Moreover, the fully faithful functor $i^\circ: DM^\circ
\rightarrow \DMeff$ induces an equivalence $DM^\circ \cong i^\circ (DM^\circ)
\cong \DMeff \cap \northogonal{1}$ which is a quasi-inverse to $\nu _{\leq 0}
\circ \iota$.

\item \label{lem.KScomp1.c}  The natural transformation $\iota \rightarrow i^{\circ} \circ
(\nu_{\leq 0} \circ \iota)$ deduced from the unit of the adjunction $(\nu _{\leq 0}, i^\circ)$
is an isomorphism.
\end{enumerate}
\end{lem}
\begin{proof}
\eqref{lem.KScomp1.a}:  This follows from the definition of $DM^{\circ}$ 
\cite[Def. 3.2.1]{Kahn:2015qf} and \cite[Thm. 9.1.16]{MR1812507}.

\eqref{lem.KScomp1.b}:  Since $i^{\circ}: DM^{\circ}\rightarrow \DMeff$
is fully faithful and a right adjoint of $\nu _{\leq 0}$, it suffices
to show that $i^{\circ}E$ is in $\northogonal{1}$, . 
Now, we observe that by construction $DM^{\circ}$ is the localization for
the pair $\DMeff (1)\subseteq \DMeff$ \cite[Def. 9.1.1]{MR1812507}.  The result then follows
from \cite[Lem. 9.1.2]{MR1812507}.

\eqref{lem.KScomp1.c}:  Since $i^\circ(DM^\circ)=\DMeff \cap \northogonal{1}$,
it suffices to show that $i^\circ \rightarrow i^{\circ} \circ (\nu_{\leq 0} \circ i^\circ)$
is an isomorphism which is clear since $i^\circ$ is fully faithful \eqref{rmk.counitiso}.
\end{proof}

 \begin{defi}  \label{def.KSunrcoh}
The Kahn-Sujatha unramified cohomology functor $R_{nr}:\DMeff \rightarrow DM^{\circ}$ is
the right adjoint of $i^{\circ}:DM^{\circ} \rightarrow \DMeff$ 
\cite[\S 5]{Kahn:2015qf}.
\end{defi}

Recall that $i_{0}$ is the inclusion $\DMeff \rightarrow DM$, which admits a right adjoint
$r_{0}$ and that $f_{0}=i_{0}\circ r_{0}$ \eqref{subsec.ssfil}.

\begin{prop}
The functor $i^{\circ}\circ R_{nr}:\DMeff \rightarrow \DMeff$ is isomorphic to 
$(r_{0}\circ bc_{\leq 0})\circ i_{0}:\DMeff \rightarrow \DMeff$, where $bc_{\leq 0}$ is
the birational cover defined in \ref{def.birat.cover}.
\end{prop}
\begin{proof}
By \ref{lem.KScomp1}\eqref{lem.KScomp1.b}, $(i^{\circ}\circ R_{nr})E\in \DMeff \cap
\northogonal{1}$ for every $E\in \DMeff$.  Hence, it suffices to show that the
natural transformation $i_0 \circ i^\circ \circ R_{nr}\rightarrow i_0$ deduced
from the counit of the adjunction $(i^{\circ}, R_{nr})$ satisfies the
universal property of \ref{lem.comp.orteff}\eqref{lem.comp.orteff.b}.

Let $E'\in \DMeff \cap \northogonal{1}$.  Then:
\[ \Hom _{DM}(i_0 E', i_0\circ i^{\circ}\circ R_{nr}E) \cong
\Hom _{\DMeff}(E', (i^{\circ}\circ R_{nr})E)
\] 
and by adjointness:
\[ \Hom _{\DMeff}(E', (i^{\circ}\circ R_{nr})E)\cong \Hom _{DM^{\circ}}(\nu _{\leq 0}E', R_{nr}E)
\cong \Hom _{\DMeff}(i^{\circ}\nu _{\leq 0}E', E)
\]
On the other hand  
$\Hom _{\DMeff}(i^{\circ}\nu _{\leq 0}E', E)\cong \Hom _{\DMeff}(E',E)$ by
\ref{lem.KScomp1}\eqref{lem.KScomp1.c}.  Hence, we conclude that
$\Hom _{DM}(i_0 E', i_0\circ i^{\circ}\circ R_{nr}E) \cong \Hom _{DM}(i_0 E',i_0 E)$.
This finishes the proof.
\end{proof}

\section{The Birational Tower for Motivic Cohomology}  \label{sec.biratHZ}

We will study the  birational tower \eqref{F.birtow} for the motive of a point $\hr$ with
coefficients in a commutative ring $R$.  Our goal is to show that the tower induces a
 finite filtration on the Chow groups, and that it satisfies several of the properties
of the still conjectural Bloch-Beilinson-Murre filtration.
						
\subsection{Basic properties}

\begin{lem}  \label{orth.HZ}
$\hr$ belongs to $\northogonal{1}$ and to $\DMeff$ \eqref{subsubsec.neffDM}.
\end{lem}
\begin{proof}
It follows directly from the definition of $\DMeff$ \eqref{subsubsec.neffDM} that
$\hr$ belongs to $\DMeff$.  It only remains to show that $\hr \in \northogonal{1}$.
Combining \ref{subsubsec.neffDM} and \ref{rmk.orthchk}, we conclude that it
suffices to see that for every $X\in Sm_{k}$ and every $p$, $q\in \mathbb Z$
with $p\geq 1$ \eqref{eq.DMeffgenstw}: $\Hom _{DM}(M(X)(p)[q],\hr)=0$.

We observe that $M(X)(p)[q]$ and $\hr$ are in $\DMeff$, thus 
by Voevodsky's cancellation theorem \cite{MR2804268} (see \ref{subsubsec.cancel})
it suffices to show that:
\[ \Hom _{\DMeff}(M(X)(p)[q],\hr)=0.
\]
Let $U=\mathbb A ^{p}
\backslash \{0\}$ and consider the Gysin distinguished triangle in $\DMeff$
\cite[15.15]{MR2242284}: $M(X\times U)\rightarrow M(X\times \mathbb A ^{p})
\rightarrow M(X)(p)[2p]$.  Hence, it suffices to show that the map induced by the open
immersion $X\times U\rightarrow X\times \mathbb A^{p}$:
$\Hom _{\DMeff}(M(X\times \mathbb A ^{p})[q],\hr)\rightarrow \Hom _{\DMeff}
(M(X\times U)[q],\hr)$ is an isomorphism for every $q\in \mathbb Z$.  But this follows
directly from the computation of motivic cohomology in weight zero 
\cite[4.2]{MR2242284}.
\end{proof}

\begin{prop}  \label{prop.hzgeq0birat}
Let $n\geq 0$ be an arbitrary integer.  Then the natural map (see \ref{F.birtow}),
$	\theta _{n}^{\hr}: bc_{\leq n}(\hr)\rightarrow \hr$ is an isomorphism in $DM$.
\end{prop}
\begin{proof}
Combining \ref{cor.comp.tool1} and \ref{eq.birtower}, it suffices to show that $\hr$ is in 
$\northogonal{1}$.  This follows from \ref{orth.HZ}.
\end{proof}

Hence, we conclude that the birational tower \eqref{F.birtow} for the motive of a point
with $R$ coefficients is as follows:

\begin{align}	\label{motbirtower}
\xymatrix@C=1.5pc{\cdots \ar[r]
		& bc_{\leq -3}(\hr) \ar[drr]|{\theta _{-3}^{\hr}} \ar[r]
		& bc_{\leq -2}(\hr) \ar[r] \ar[dr]|{\theta _{-2}^{\hr}}& bc_{\leq -1}(\hr) 
		\ar[d]|{\theta _{-1}^{\hr}}  \\						
		&&& \hr}
\end{align}

\subsection{The filtration on the motivic cohomology of a scheme}

Let $X\in Sm_{k}$, and $p$, $q\in \mathbb Z$.  We will write $H^{p,q}(X,R)$ for 
$\Hom _{DM}(M(X), \hr (q)[p])$, i.e. for the motivic cohomology of $X$ with coefficients in 
$R$ of degree $p$ and weight $q$.  By adjointness, $H^{p,q}(X,R)\cong \Hom _{DM}
(M(X)(-q)[-p],\hr)$.
Therefore, the abutment of
the spectral sequence \eqref{thm.birspecseq} for the tower \eqref{motbirtower} evaluated in 
$M(X)(-q)[-p]$ induces a filtration on $H^{p,q}(X,R)$.
Namely:

\begin{defi}  \label{thm.filt.motcohY}
Let $p$, $q$ be arbitrary integers, and let $X$ be in $Sm_{k}$.  Consider the
decreasing filtration $F^{\bullet}$ on $H^{p,q}(X,R)$ where 
$F^n H^{p,q}(X,R)$ is given by
the image of $\theta _{-n\ast}^{\hr}$ (see \ref{motbirtower}), $n\geq 0$:
\[ \xymatrix{\Hom _{DM}(M(X)(-q)[-p], bc_{\leq -n}\hr)
\ar[d]^-{\theta _{-n\ast}^{\hr}}\\
\Hom _{DM}(M(X)(-q)[-p], \hr)=H^{p,q}(X,R).} \]
By construction, the filtration $F^{\bullet}$ is functorial in $X$ with respect to
morphisms in $DM$.
\end{defi}

\begin{rmk}  \label{rmk.funcfil}
The existence of the functor $Chow^{\mathrm{eff}}(k)\rightarrow DM$ 
\cite[Prop. 2.1.4, Thm. 3.2.6]{MR1764202},
where $Chow^{\mathrm{eff}}(k)$ is the category of effective Chow motives over $k$;
implies that the filtration $F^\bullet$ constructed in \ref{thm.filt.motcohY}
is functorial with respect to Chow correspondences. 
\end{rmk}

\subsubsection{Finiteness of the filtration}

Our goal is to show that for $X$ in $Sm_{k}$, the filtration $F^{\bullet}$ on $H^{p,q}(X,R)$
defined in \ref{thm.filt.motcohY} is concentrated in the range $0\leq n\leq q$.

\begin{thm}  \label{thm.filt.motcoh.finite}
Let $p$, $q$ be arbitrary integers, and let $X$ be in $Sm_{k}$.  Then the decreasing
filtration $F^{\bullet}$ on $H^{p,q}(X,R)$ constructed in \ref{thm.filt.motcohY} satisfies the
following properties:
\begin{enumerate}
	\item \label{thm.filt.motcoh.finite.a}  $F^{0}H^{p,q}(X,R)=H^{p,q}(X,R)$,
	\item \label{thm.filt.motcoh.finite.b}  $F^{q+1}H^{p,q}(X,R)=0$.
\end{enumerate}
\end{thm}
\begin{proof}
By \ref{prop.hzgeq0birat} we deduce that $\theta ^{\hr}_{0}$ is an isomorphism in 
$DM$.  This proves the first claim.  For the second claim, we observe that 
$M(X)(-q)[-p]$ is in $\DMeff (-q)$ (see \ref{eq.DMeffgenstw} and \ref{subsubsec.neffDM}).
Thus, it suffices to show that $bc_{\leq -q-1}\hr$ belongs to $\northogonal{-q}$ (see
\ref{subsubsec.neffDM}).
This follows from \ref{prop.adj-ort.b} and \ref{def.birat.cover}.
\end{proof}

\subsection{The components of the filtration}

We will describe some properties that an element in $H^{p,q}(X,R)$
needs to satisfy in order to be in the $n$-component of the filtration $F^{\bullet}$
\eqref{thm.filt.motcohY}.  First, we fix some notation.

\subsubsection{}  \label{subsubsec.notpre}

Let $X\in Sm _k$ and $\alpha \in H^{p,q}(X,R)$ with $q\geq 0$.
Set $A=M(X)(-q)[-p]\in DM$.  Then $\alpha$ induces a map 
$\alpha(-q)[-p]:A\rightarrow \sphere _R$.  Let $n>0$ and set $n'=-n+1$.
Recall the $n'$-effective
cover $\epsilon _{n'}^{A}:f_{n'}(A)\rightarrow A$
\eqref{rmk.univeffcar}.

\begin{prop}  \label{prop.crit1}
With the notation of \ref{subsubsec.notpre}, $\alpha \in F^{n}H^{p,q}(X,R)$
(see \ref{thm.filt.motcohY}) if and only if $\alpha (-q)[-p]\circ \epsilon _{n'}^{A}=0$.
\end{prop}
\begin{proof}
($\Rightarrow$):  By construction \eqref{thm.filt.motcohY}, it suffices to show that:
\[  \Hom _{DM}(f_{-n+1}A, bc_{\leq -n}\sphere _R)=0.
\]  
This follows directly from
\ref{rmk.univcar} and \ref{rmk.univeffcar} (see \ref{subsubsec.neffDM}).

($\Leftarrow$):  Let $s_{<n'}(A)$ be the cone of $\epsilon _{n'}^A$, and consider the
following diagram in $DM$ where the top row is a distinguished triangle:
\[  \xymatrix{f_{n'}(A)[1] & \ar[l] s_{<n'}(A) \ar@{..>}_-{\alpha ''}[d] \ar@{-->}_-{\alpha '}[dr]
& \ar[l] \ar[d]^-{\alpha (-q)[-p]}  A &  \ar[l]_-{\epsilon _{n'}^{A}} f_{n'}(A)\\
& bc_{\leq -n}(\sphere _R) \ar_-{\theta _{-n}^{\sphere _R}}[r]& \sphere _R &}
\]
Hence, $\alpha (-q)[-p]\circ \epsilon _{n'}^{A}=0$ if and only if there exists $\alpha '$
which makes the right triangle commute.
On the other hand, the universal property of $\epsilon _{n'}^A$ \eqref{rmk.univeffcar}
implies that $s_{<n'}A\in DM^{\perp}(-n+1)$ (see \ref{subsubsec.neffDM}).  Thus, by
the universal property of $\theta _{-n}^{\sphere _R}$ 
\eqref{counit-properties}-\eqref{rmk.univcar} we conclude that the existence of $\alpha '$
such that the right triangle commutes is equivalent to the existence of $\alpha ''$
such that the square commutes.

Hence, by construction \eqref{thm.filt.motcohY} we conclude that
$\alpha \in F^{n}H^{p,q}(X,R)$.
\end{proof}

\begin{cor}  \label{cor.crit2}
With the notation of \ref{subsubsec.notpre}.  Then $F^n H^{p,q}(X,R)$ 
(see \ref{thm.filt.motcohY}) is isomorphic to:
\[ \mathrm{Ker}(\Hom_{DM}(M(X)(-q)[-p],\sphere _R)
\stackrel{\epsilon ^{A \; \ast}_{-n+1}}{\longrightarrow} 
\Hom_{DM}(f_{-n+1}(M(X)(-q)[-p]),\sphere _R))
\]
which is also isomorphic to:
\[ \mathrm{Ker}(\Hom_{DM}(M(X),\sphere _R (q)[p])
\stackrel{\epsilon ^{M(X)\; \ast}_{q-n+1}}{\longrightarrow} 
\Hom_{DM}(f_{q-n+1}(M(X)),\sphere _R(q)[p]))
\]
\end{cor}
\begin{proof}
The first isomorphism follows directly from \ref{prop.crit1};  the second one
follows by adjointness and \ref{prop.Tatetwist}\eqref{prop.Tatetwist.b}.
\end{proof}

\begin{rmk}  \label{rmk.nobc}
By \ref{cor.crit2} it is possible to construct the filtration $F^\bullet$
\eqref{thm.filt.motcohY} using only the effective covers of the slice filtration. 
\end{rmk}

\subsubsection{The filtration for smooth projective varieties}  \label{Fil.smproj}

We will show that  an algebraic cycle $\alpha \in H^{2q,q}(X,R)\cong CH^q(X)$
is in $F^1 H^{2q,q}(X,R)$ \eqref{thm.filt.motcohY} if and only if $\alpha$ is
numerically equivalent to zero.  When the coefficient ring $R\neq \mathbb Q$,
we will say that $\alpha$ is numerically equivalent to zero if for every field
extension $K/k$ of finite transcendence degree and every cycle
$\beta \in CH_q(X_K)_R$, the intersection multiplicity 
$\deg(\alpha _K\cdot \beta)\in R$ is zero.

Let $\inthomeff$ denote the internal $\Hom$-functor in $\DMeff$.  Notice that it
does not coincide in general with the internal $\Hom$-functor in $DM$ (see the
remark after Cor. 4.3.6 in \cite{MR1764202}).

The following result is based on the work of Kahn-Sujatha 
\cite[Lem. 5.1 and its proof]{Kahn:2015kq}.  The fact
that this approach does in fact compute the first component of the filtration was
brought to my attention by the anonymous referee. 

\begin{thm}  \label{thm.main2}
Let $X$ be a smooth projective $k$-scheme and $q\geq 0$.  
Then via the isomorphism $H^{2q,q}(X,R)\cong CH^q(X)_R$, 
$F^1H^{2q,q}(X,R)$ \eqref{thm.filt.motcohY} gets identified with the $R$-submodule
of cycles numerically equivalent to zero.
\end{thm}
\begin{proof}
Let $\alpha \in CH^q(X)_R$.
By \ref{cor.crit2}, it suffices to show that
$\alpha$ is numerically equivalent to zero if and only if the composition:
\[  \xymatrix{f_qM(X)\ar[r]^-{\epsilon _q ^{M(X)}}& 
M(X) \ar[r]^-{\alpha}& \sphere _R (q)[2q]}
\]
is zero in $DM$, or equivalently in $\DMeff$ since $f_qM(X)$, $M(X)$,
$\sphere _R (q)[2q]$ are in $\DMeff$.  By \cite[Prop. 1.1]{MR2249535},
$f_qM(X)\cong \inthomeff (\sphere (q)[2q], M(X))(q)[2q]$, so we are reduced
to show that $\alpha$ is numerically equivalent to zero if and only if
the composition:
\[  \xymatrix{\inthomeff (\sphere (q)[2q], M(X))(q)[2q]\ar[r]^-{\epsilon _q ^{M(X)}}& 
M(X) \ar[r]^-{\alpha}& \sphere _R (q)[2q]}
\]
is zero in $\DMeff$, which is equivalent (by adjointness and Voevodsky's cancellation
theorem \cite{MR2804268}) to show that the map induced by $\alpha$:
\[  \xymatrix{\inthomeff (\sphere (q)[2q], M(X))\ar[r]^-{\alpha _\ast}& 
 \inthomeff(\sphere _R (q)[2q],\sphere _R (q)[2q])\cong \sphere _R}
\]
is zero.  It follows from \cite[Prop. 2.3]{Kahn:2015kq} that
$\inthomeff (\sphere (q)[2q], M(X))$ is in $(\DMeff )_{\geq 0}$ for Voevodsky's
homotopy $t$-structure \cite[p. 11]{MR1764202}, and by \cite[4.2]{MR2242284}
$\sphere _R$ is in the heart of the homotopy $t$-structure.  Hence,
$\alpha _\ast$ descends to a map:
\[  \xymatrix{h_0(\inthomeff (\sphere (q)[2q], M(X)))\ar[r]^-{\bar{\alpha} _\ast}& \sphere _R}
\]
which is zero if and only if $\alpha _\ast$ is zero.  Now,
$\bar{\alpha}_\ast$ is a morphism in the heart of the homotopy $t$-structure,
which is given by homotopy invariant Nisnevich sheaves with transfers.  Thus,
$\bar{\alpha}_\ast$ is zero if and only if the induced map on sections
\cite[11.2]{MR2242284}:
\[  \xymatrix{\Gamma(K,h_0(\inthomeff (\sphere (q)[2q], M(X))))\ar[r]^-{\bar{\alpha} _{K\ast}}
& \Gamma(K,\sphere _R})\cong R
\]
is zero for every field extension $K/k$ of finite transcendence degree.
By \cite[Prop. 2.3]{Kahn:2015kq},
$\Gamma(K,h_0(\inthomeff (\sphere (q)[2q], M(X))))
\cong CH_q(X_K)$; so we are reduced to show that $\bar{\alpha} _{K\ast}$
computes  the intersection multiplicity with $\alpha _K$.

Let $Y\in Sm_k$ with function field $K$.  Then, by adjointness $\bar{\alpha} _{K\ast}$
is the colimit indexed by the non-empty open subsets  $U\subseteq Y$:
\[  \xymatrix{\Hom _{\DMeff} (M(U)\sphere (q)[2q], M(X))\ar[r]^-{\alpha _{U \ast}}& 
 \Hom _{\DMeff}(M(U)(q)[2q],\sphere _R (q)[2q]) \ar[d]_-{\cong}\\
 & CH^0(U)_R\cong R}
\]
Since $X$ is smooth projective, $\Hom _{\DMeff} (M(U)\sphere (q)[2q], M(X))
\cong CH_{d_U+q}(U\times X)_R$  where $d_u$ is the dimension of $U$.
Hence, it suffices to show that $\alpha _{U \ast}$ gets identified with the map
$\beta \mapsto p_{U\ast}(\beta \cdot (U\times \alpha))$, $\beta \in 
CH_{d_U+q}(U\times X)_R$ where $p_U:U\times X \rightarrow U$.
 
Finally, this follows by combining the existence of the functor
$Chow^{\mathrm{eff}}(k)\rightarrow DM$ 
\cite[Prop. 2.1.4, Thm. 3.2.6]{MR1764202}
(where $Chow^{\mathrm{eff}}(k)$ is the category of effective Chow motives over $k$)
with the projective bundle formula in $DM$ \cite[15.12]{MR2242284} (which shows in
particular that $M(\mathbb P^q)/M(\mathbb P^{q-1})$ is a model for $\sphere (q)[2q]$
\cite[15.2]{MR2242284}).
\end{proof}

\section{Towards the Bloch-Beilinson-Murre filtration}  \label{subsec.BBMfil}

In this section we include some remarks on the still conjectural Bloch-Beilinson-Murre
filtration \cite{MR1265533}, \cite{MR1389964}, \cite{MR1744952}, \cite{MR1744947}.

\subsection{Basic Setup}  \label{subsec.bassetup}
Recall that $k$ is a perfect base field.
Let $SmProj_k$ be the full subcategory of $Sm_{k}$ where the objects are
smooth projective varieties over $k$.  We will write $CH^{\ast}(X)_{\mathbb Q}$ for the Chow
ring of $X$ with rational coefficients: $\oplus_{n=0}^{\dim X} CH^{n}(X)_{\mathbb Q}$.
We fix a Weil cohomology theory $H$ on $SmProj_k$ \cite[1.2]{MR0292838}, with cycle
class map $\mathrm{cl}_{X}:CH^{p}(X)_{\mathbb Q}\rightarrow H^{2p}(X)$.
Given a correspondence $\Lambda \in CH^{q}(X\times Y)_{\mathbb Q}$, we will write
 respectively
$CH_{\Lambda}^{i}:CH^{i+d_{X}-q}(X)_{\mathbb Q}\rightarrow CH^{i}(Y)_{\mathbb Q}$,
$H_{\Lambda}^{i}:H^{i+2(d_{X}-q)}(X)\rightarrow H^{i}(Y)$ for the 
induced maps on the Chow groups and on cohomology \cite[1.3]{MR0292838}, 
where $d_{X}$ is the dimension of $X$.  Namely, 
$CH_{\Lambda}^{i}(\alpha)=\pi _{Y \ast}(\Lambda \cdot 
\pi _{X}^{\ast}\alpha)$, where $\pi _X$, $\pi _Y$ are the relative projections to
$X$ and $Y$; and similarly for $H_\Lambda ^i$.

\begin{defi}  \label{def.BBMfil}
Let $X\in SmProj_k$. 
We will say that a descending filtration $F^{\bullet}$
on $CH^{\ast}(X)_{\mathbb Q}$ is a filtration of Bloch-Beilinson-Murre type if the following
conditions are satisfied:
\begin{enumerate}[{\bf BBM1:}]
\item \label{def.BBMfil.a}
$F^{0}CH^{q}(X)_{\mathbb Q}=CH^{q}(X)_{\mathbb Q}$, for every 
	$0\leq q\leq \dim X$.
\item  \label{def.BBMfil.b}
$F^{1}CH^{q}(X)_{\mathbb Q}=CH^{q}_{num}(X)_{\mathbb Q}$, the group
of cycles numerically equivalent to zero modulo rational equivalence.
\item  \label{def.BBMfil.c}
$F^{\bullet}$ is functorial with respect to Chow correspondences between
smooth projective varieties, i.e. given $\Lambda \in CH^{q'}(X\times Y)_{\mathbb Q}$:
\[ CH_{\Lambda}^{q}(F^{n}CH^{q+d_{X}-q'}(X)_{\mathbb Q})\subseteq F^{n}
	CH^{q}(Y)_{\mathbb Q}.\]
We will write $Gr^{n}_{F}(CH_{\Lambda}^{q})$ for the map induced by $CH_{\Lambda}^{q}$
on the graded groups:
\[Gr^{n}_{F}(CH^{q+d_{X}-q'}(X)_{\mathbb Q})\rightarrow Gr^{n}_{F}(CH^{q}
(Y)_{\mathbb Q}).
\]
\item \label{def.BBMfil.d}
With the notation of {\bf BBM\ref{def.BBMfil.c}} and \ref{subsec.bassetup}.  
If $H_{\Lambda}^{2q-n}=0$ then $Gr^{n}_{F}(CH_{\Lambda}^{q})=0$.
\item \label{def.BBMfil.e}
$F^{q+1}CH^{q}(X)_{\mathbb Q}=0$.	
\end{enumerate}
\end{defi}

\subsubsection{} \label{rmk.relaxhomcond}
Instead of {\bf BBM\ref{def.BBMfil.d}} it is possible to consider the weaker condition
{\bf BBM\ref{def.BBMfil.d}a}: If the class of $\Lambda \in CH^{q'}(X\times Y)_{\mathbb Q}$
is homologous to zero, i.e. $\mathrm{cl}_{X\times Y}(\Lambda)=0\in H^{2q'}(X\times Y)$;
then $Gr^{n}_{F}(CH_{\Lambda}^{q})=0$ for every $n$; $0\leq n\leq q$ (see the graded
condition in \cite[p. 421]{MR2082666}).

\begin{rmk}  \label{rmk.ssconj}
Assuming that homological and numerical equivalence coincide on $X$, i.e.
the standard conjecture $D(X)$ \cite[Prop. 3.6]{MR0292838}; we observe that
{\bf BBM\ref{def.BBMfil.b}} simply says that $F^{1}CH^{q}(X)_{\mathbb Q}$ is given by
the group of cycles homologically equivalent to zero modulo rational
equivalence $CH^{q}_{hom}(X)_{\mathbb Q}$.  This is
the condition that appears in \cite{MR1744952}, \cite{MR1744947}.
Notice that when the base field $k=\mathbb C$ is given by the complex numbers, then
$D(X)$ is known to be true in the following cases:
\begin{enumerate}
\item $CH^{q}_{hom}(X)_{\mathbb Q}=CH^{q}_{num}(X)_{\mathbb Q}$; for
$q=0$, $1$, $2$, $\dim X -1$, $\dim X$ \cite[p. 369 Cor. 1]{MR0230336}.  In fact,
when $q=1$ (resp. $q=\dim X$) and $k$ is algebraically closed,  homological
and numerical equivalence coincide \cite{MR0082730} (resp. 
non-triviality of the cycle map \cite[p.363 1.2C(iii)]{MR0292838}).
\item $X$ is an abelian variety \cite[p. 372 Thm. 4]{MR0230336},
\item $\dim X\leq 4$ \cite[p. 369 Cor. 1]{MR0230336}.
\end{enumerate}
\end{rmk}

Since $CH^{q}(X)_{\mathbb Q}\cong \Hom _{DM}(M(X)(-q)[-2q], \sphere _{\mathbb Q})$,
we can restrict the filtration $F^{\bullet}$ considered in \eqref{thm.filt.motcohY} to the
Chow groups with rational coefficients in order to obtain a filtration that satisfies
several of the properties considered in \eqref{def.BBMfil}:

\begin{thm}  \label{thm.BBMfilcond}
With the notation of \ref{subsec.bassetup}.
Let $X\in SmProj_k$.  Then the filtration $F^\bullet$  on $CH^{\ast}(X)_{\mathbb Q}$
defined in \eqref{thm.filt.motcohY}
satisfies the conditions
{\bf BBM\ref{def.BBMfil.a}}, {\bf BBM\ref{def.BBMfil.b}}, {\bf BBM\ref{def.BBMfil.c}} and
{\bf BBM\ref{def.BBMfil.e}} of \eqref{def.BBMfil}.
\end{thm}
\begin{proof}
In order to obtain the filtration we just take $R=\mathbb Q$ in \eqref{thm.filt.motcohY},
and observe that $CH^{q}(X)_{\mathbb Q}\cong \Hom _{DM}(M(X)(-q)[-2q], 
\sphere _{\mathbb Q})$.

Now, property {\bf BBM\ref{def.BBMfil.a}} follows from 
\ref{thm.filt.motcoh.finite}\eqref{thm.filt.motcoh.finite.a}, property {\bf BBM\ref{def.BBMfil.b}}
follows from \ref{thm.main2}, property 
{\bf BBM\ref{def.BBMfil.c}} follows from \ref{thm.filt.motcohY}-\ref{rmk.funcfil},
and property {\bf BBM\ref{def.BBMfil.e}} follows from
\ref{thm.filt.motcoh.finite}\eqref{thm.filt.motcoh.finite.b}.
\end{proof}

\begin{rmk}  \label{rmk.noQcond}
Notice that the filtration defined in \ref{thm.filt.motcohY} satisfies the conditions required
in \ref{thm.BBMfilcond} for any coefficient ring $R$ and not just the rational numbers.
\end{rmk}

\subsubsection{}  \label{conjmult.homcond}
In general there is not much that can be said at the moment with respect to condition
{\bf BBM\ref{def.BBMfil.d}}.  However, assuming that the tensor structure in $DM$
is compatible with the birational covers $bc_{\leq -n} \sphere _R \rightarrow \sphere _R$
in the sense of \ref{conj.mulprop}, then it is possible to show that 
{\bf BBM\ref{def.BBMfil.d}a} holds \eqref{prop.conj.homo}.

\begin{conj}  \label{conj.mulprop}
Consider the tower \eqref{motbirtower} in $DM$ for $R=\mathbb Q$, and let
$n$, $m\geq 0$.  Then there exists a map $\mu _{n,m}:(bc_{\leq -n} \sphere _{\mathbb Q})
\otimes (bc_{\leq -m} \sphere _{\mathbb Q})\rightarrow bc_{\leq -n-m}\sphere _{\mathbb Q}$,
such that the following diagram commutes:
\[  
\xymatrix{(bc_{\leq -n} \sphere _\mathbb Q)\otimes (bc_{\leq -m} \sphere _\mathbb Q) 
\ar[rr]^-{\theta^{\sphere _{\mathbb Q}}_{-n}\otimes \theta^{\sphere _{\mathbb Q}}_{-m}}
\ar[d]_-{\mu _{n,m}} && 
(\sphere _{\mathbb Q} \otimes \sphere _{\mathbb Q}) \cong \sphere _{\mathbb Q} \\
	bc_{\leq -n-m}\sphere _{\mathbb Q} \ar[urr]_-{\theta^{\sphere _{\mathbb Q}}_{-n-m}}&&}
\]
\end{conj}

\begin{prop}  \label{prop.conj.homo}
Let $X\in SmProj_k$.
Assume that the maps $\mu_{n,1}$ in \ref{conj.mulprop} exist for $n\geq 0$.  
Then the filtration $F^\bullet$  on $CH^{\ast}(X)_{\mathbb Q}$
considered in \eqref{thm.filt.motcohY} and \eqref{thm.BBMfilcond}
satisfies the remaining condition {\bf BBM\ref{def.BBMfil.d}a} \eqref{rmk.relaxhomcond}.
\end{prop}
\begin{proof}
Let $Y \in SmProj_k$, and $\Lambda \in CH^{q'}(X\times Y)_{\mathbb Q}$ a
correspondence homologous to zero, in particular numerically equivalent
to zero by \cite[Prop. 1.2.3]{MR0292838}.  Consider $\Lambda$ as a map
$M(X\times Y)\rightarrow \sphere _{\mathbb Q}(q')[2q']$.  By \ref{thm.main2},
we conclude that there exists a map
$\Lambda ':M(X\times Y)(-q')[-2q']\rightarrow bc_{\leq -1}\sphere _{\mathbb Q}$ in
$DM$ such that $\Lambda (-q')[-2q']=\theta ^{\sphere _{\mathbb Q}}_{-1}\circ \Lambda '$
\eqref{motbirtower}.

Now let $\alpha \in F^{n}CH^{q}(X)_{\mathbb Q}$, which we will consider as a map
$\alpha:M(X)\rightarrow \sphere _{\mathbb Q}(q)[2q]$.  Hence, by construction
\eqref{thm.filt.motcohY} there exists a map $\alpha ':M(X)(-q)[-2q]\rightarrow bc_{\leq -n}
\sphere _{\mathbb Q}$, such that $\alpha (-q)[-2q] 
=\theta ^{\sphere _{\mathbb Q}}_{-n}\circ \alpha '$
\eqref{motbirtower}.

Thus, it
suffices to show that the following composition factors through
$\theta^{\sphere _{\mathbb Q}}_{-n-1}:bc_{\leq -n-1}\sphere _{\mathbb Q}\rightarrow
\sphere _{\mathbb Q}$ (see \ref{thm.filt.motcohY}):
\[\xymatrix@=1.1pc{(M(X)(-q)[-2q])\otimes(M(X)\otimes M(Y))(-q')[-2q'] \ar[d]_-{\alpha'
\otimes \Lambda '}& \\
(bc_{\leq -n} \sphere _{\mathbb Q})\otimes (bc_{\leq -1} \sphere _{\mathbb Q}) \ar[r]^-{\theta
^{\sphere _{\mathbb Q}}_{-n}\otimes \theta ^{\sphere _{\mathbb Q}}_{-1}}&
(\sphere _{\mathbb Q} \otimes \sphere _{\mathbb Q})\cong \sphere _{\mathbb Q}}
\]
But this follows directly from the existence of the map $\mu _{n,1}$ \eqref{conj.mulprop}.
\end{proof}

\subsubsection{}  
The anonymous referee suggests the following heuristic argument to explain why
the filtration considered in \ref{thm.BBMfilcond} is a reasonable candidate for a 
Bloch-Beilinson-Murre filtration:  Let $X\in SmProj_k$ be of pure dimension $d$.
Combining \ref{prop.Tatetwist} and \ref{prop.crit1} we conclude that 
$\alpha \in CH^{q}(X)_{\mathbb Q}$ is in $F^{n}CH^{q}(X)_{\mathbb Q}$ 
if and only if the composition:
\[  \xymatrix{f_{q-n+1}(M(X)) \ar[r]^-{\epsilon _{q-n+1}^{M(X)}}& M(X) 
\ar[r]^-{\alpha} & \sphere _R (q)[2q]}
\]
is zero in $\DMeff$.  Let $r=q-n+1$, by \cite[Prop. 1.1]{MR2249535}, 
\cite[Lem. 5.9]{MR2600283} we conclude that $f_{r}(M(X))\cong \inthomeff
(\sphere _{\mathbb Q}(r)[2r], M(X))\otimes \sphere _{\mathbb Q}(r)[2r]$ in $\DMeff$. 
Thus, it follows from Voevodsky's cancellation theorem \cite{MR2804268} that
$\alpha \in F^nCH^q(X)_{\mathbb Q}$ if and only if the map
$\inthomeff(\sphere _{\mathbb Q}(r)[2r], M(X))\rightarrow \sphere_{\mathbb Q}(n-1)[2n-2]$
in $\DMeff$ induced by $\alpha$ is zero.

Assuming a conjecture of Ayoub \cite[4.22]{ayoub:preprint},  the functor
$\DMeff \rightarrow \DMeff$; $A\mapsto \inthomeff(\sphere _{\mathbb Q}(1),A)$ 
maps $\DMeff _{\leq n}$
to $\DMeff _{\leq n-1}$, where $\DMeff _{\leq n}$ is the smallest full triangulated
subcategory of $\DMeff$ which is closed under arbitrary (infinite) coproducts and contains
$M(Y)$ for $Y\in Sm_k$, $\dim Y\leq n$.  Thus, we conclude that 
$\inthomeff(\sphere _{\mathbb Q}(r)[2r], M(X))$ is in 
$\DMeff _{\leq d-q+n-1}$.

Hence, using the argument in \ref{thm.main2} we conclude that $F^nCH^q(X)_{\mathbb Q}$
is contained in the kernel of the maps $CH^{n-1}_{\Lambda}:CH^q(X)_{\mathbb Q}
\rightarrow CH^{n-1}(Y)_{\mathbb Q}$; $\alpha \mapsto \pi _{Y \ast}(\Lambda \cdot 
\pi _{X}^{\ast}\alpha)$ (see \ref{subsec.bassetup}), where $\Lambda \in CH^{d-q+n-1}
(Y\times X)
_{\mathbb Q}$ and $\dim Y\leq d-q+n-1$.  This implies that
$F^nCH^q(X)_{\mathbb Q}$ is contained in the filtration defined in \cite[\S 1.1]{MR2082666}
which for zero cycles is a natural candidate for a Bloch-Beilinson-Murre filtration
\cite[Prop. 6]{MR2082666}.

For more details on Ayoub's conjecture \cite[4.22]{ayoub:preprint}, we refer the reader
to the recent preprint \cite{bondarko:prep}.


\section*{Acknowledgements}
The author would like to thank warmly Chuck Weibel for his interest in this work, as well as
for all the advice and support during all these years.  The author also thanks the anonymous
referee for bringing to our attention the work of Kahn-Sujatha 
\cite{Kahn:2015kq} on unramified cohomology and their computation
of the first component of the filtration studied in this paper.  
The author would also like to thank Bruno Kahn,
Mikhail Bondarko and Jinhyun Park for several conversations on the 
Bloch-Beilinson-Murre filtration.


\bibliography{biblio_locSH}

\begin{thebibliography}{10}

\bibitem{MR2438151}
J.~Ayoub.
\newblock Les six op\'erations de {G}rothendieck et le formalisme des cycles
  \'evanescents dans le monde motivique. {II}.
\newblock {\em Ast\'erisque}, (315):vi+364 pp. (2008), 2007.

\bibitem{MR2735752}
J.~Ayoub.
\newblock The {$n$}-motivic {$t$}-structures for {$n=0$}, {$1$} and {$2$}.
\newblock {\em Adv. Math.}, 226(1):111--138, 2011.

\bibitem{ayoub:preprint}
J.~Ayoub.
\newblock Motives and algebraic cycles: a selection of conjectures and open
  questions.
\newblock {\em Preprint}, 2015.

\bibitem{MR923131}
A.~A. Be{\u\i}linson.
\newblock Height pairing between algebraic cycles.
\newblock In {\em {$K$}-theory, arithmetic and geometry ({M}oscow,
  1984--1986)}, volume 1289 of {\em Lecture Notes in Math.}, pages 1--25.
  Springer, Berlin, 1987.

\bibitem{MR1780498}
T.~Beke.
\newblock Sheafifiable homotopy model categories.
\newblock {\em Math. Proc. Cambridge Philos. Soc.}, 129(3):447--475, 2000.

\bibitem{MR558224}
S.~Bloch.
\newblock {\em Lectures on algebraic cycles}.
\newblock Duke University Mathematics Series, IV. Duke University Mathematics
  Department, Durham, N.C., 1980.

\bibitem{bondarko:prep}
M.~Bondarko.
\newblock Intersecting the dimension fitlration with the slice one for
  (relative) motivic categories.
\newblock {\em Preprint}, 2016.

\bibitem{MR1944041}
P.~S. Hirschhorn.
\newblock {\em Model categories and their localizations}, volume~99 of {\em
  Mathematical Surveys and Monographs}.
\newblock American Mathematical Society, Providence, RI, 2003.

\bibitem{MR1860878}
M.~Hovey.
\newblock Spectra and symmetric spectra in general model categories.
\newblock {\em J. Pure Appl. Algebra}, 165(1):63--127, 2001.

\bibitem{MR2249535}
A.~Huber and B.~Kahn.
\newblock The slice filtration and mixed {T}ate motives.
\newblock {\em Compos. Math.}, 142(4):907--936, 2006.

\bibitem{MR1265533}
U.~Jannsen.
\newblock Motivic sheaves and filtrations on {C}how groups.
\newblock In {\em Motives ({S}eattle, {WA}, 1991)}, volume~55 of {\em Proc.
  Sympos. Pure Math.}, pages 245--302. Amer. Math. Soc., Providence, RI, 1994.

\bibitem{MR1744947}
U.~Jannsen.
\newblock Equivalence relations on algebraic cycles.
\newblock In {\em The arithmetic and geometry of algebraic cycles ({B}anff,
  {AB}, 1998)}, volume 548 of {\em NATO Sci. Ser. C Math. Phys. Sci.}, pages
  225--260. Kluwer Acad. Publ., Dordrecht, 2000.

\bibitem{Kahn:2015qf}
B.~Kahn and R.~Sujatha.
\newblock Birational motives, {II}: Triangulated birational motives.
\newblock 2015.

\bibitem{Kahn:2015kq}
B.~Kahn and R.~Sujatha.
\newblock The derived functors of unramified cohomology.
\newblock 2015.

\bibitem{MR0292838}
S.~L. Kleiman.
\newblock Algebraic cycles and the {W}eil conjectures.
\newblock In {\em Dix espos\'es sur la cohomologie des sch\'emas}, pages
  359--386. North-Holland, Amsterdam, 1968.

\bibitem{MR0230336}
D.~I. Lieberman.
\newblock Numerical and homological equivalence of algebraic cycles on {H}odge
  manifolds.
\newblock {\em Amer. J. Math.}, 90:366--374, 1968.

\bibitem{MR0082730}
T.~Matsusaka.
\newblock The criteria for algebraic equivalence and the torsion group.
\newblock {\em Amer. J. Math.}, 79:53--66, 1957.

\bibitem{MR2242284}
C.~Mazza, V.~Voevodsky, and C.~Weibel.
\newblock {\em Lecture notes on motivic cohomology}, volume~2 of {\em Clay
  Mathematics Monographs}.
\newblock American Mathematical Society, Providence, RI, 2006.

\bibitem{MR1225267}
J.~P. Murre.
\newblock On a conjectural filtration on the {C}how groups of an algebraic
  variety. {I}. {T}he general conjectures and some examples.
\newblock {\em Indag. Math. (N.S.)}, 4(2):177--188, 1993.

\bibitem{MR1308405}
A.~Neeman.
\newblock The {G}rothendieck duality theorem via {B}ousfield's techniques and
  {B}rown representability.
\newblock {\em J. Amer. Math. Soc.}, 9(1):205--236, 1996.

\bibitem{MR1812507}
A.~Neeman.
\newblock {\em Triangulated categories}, volume 148 of {\em Annals of
  Mathematics Studies}.
\newblock Princeton University Press, Princeton, NJ, 2001.

\bibitem{MR3035769}
P.~Pelaez.
\newblock Birational motivic homotopy theories and the slice filtration.
\newblock {\em Doc. Math.}, 18:51--70, 2013.

\bibitem{Pelaez:2014}
P.~Pelaez.
\newblock Motivic {B}irational {C}overs and {F}inite {F}iltrations on {C}how
  {G}roups.
\newblock 2014.

\bibitem{MR1389964}
S.~Saito.
\newblock Motives and filtrations on {C}how groups.
\newblock {\em Invent. Math.}, 125(1):149--196, 1996.

\bibitem{MR1744952}
S.~Saito.
\newblock Motives and filtrations on {C}how groups. {II}.
\newblock In {\em The arithmetic and geometry of algebraic cycles ({B}anff,
  {AB}, 1998)}, volume 548 of {\em NATO Sci. Ser. C Math. Phys. Sci.}, pages
  321--346. Kluwer Acad. Publ., Dordrecht, 2000.

\bibitem{MR1764199}
A.~Suslin and V.~Voevodsky.
\newblock Relative cycles and {C}how sheaves.
\newblock In {\em Cycles, transfers, and motivic homology theories}, volume 143
  of {\em Ann. of Math. Stud.}, pages 10--86. Princeton Univ. Press, Princeton,
  NJ, 2000.

\bibitem{MR1764202}
V.~Voevodsky.
\newblock Triangulated categories of motives over a field.
\newblock In {\em Cycles, transfers, and motivic homology theories}, volume 143
  of {\em Ann. of Math. Stud.}, pages 188--238. Princeton Univ. Press,
  Princeton, NJ, 2000.

\bibitem{MR1883180}
V.~Voevodsky.
\newblock Motivic cohomology groups are isomorphic to higher {C}how groups in
  any characteristic.
\newblock {\em Int. Math. Res. Not.}, (7):351--355, 2002.

\bibitem{MR1977582}
V.~Voevodsky.
\newblock Open problems in the motivic stable homotopy theory. {I}.
\newblock In {\em Motives, polylogarithms and Hodge theory, Part I (Irvine, CA,
  1998)}, volume~3 of {\em Int. Press Lect. Ser.}, pages 3--34. Int. Press,
  Somerville, MA, 2002.

\bibitem{MR1890744}
V.~Voevodsky.
\newblock A possible new approach to the motivic spectral sequence for
  algebraic {$K$}-theory.
\newblock In {\em Recent progress in homotopy theory ({B}altimore, {MD},
  2000)}, volume 293 of {\em Contemp. Math.}, pages 371--379. Amer. Math. Soc.,
  Providence, RI, 2002.

\bibitem{MR2804268}
V.~Voevodsky.
\newblock Cancellation theorem.
\newblock {\em Doc. Math.}, (Extra volume: Andrei A. Suslin sixtieth
  birthday):671--685, 2010.

\bibitem{MR2600283}
V.~Voevodsky.
\newblock Motives over simplicial schemes.
\newblock {\em J. K-Theory}, 5(1):1--38, 2010.

\bibitem{MR2082666}
C.~Voisin.
\newblock Remarks on filtrations on {C}how groups and the {B}loch conjecture.
\newblock {\em Ann. Mat. Pura Appl. (4)}, 183(3):421--438, 2004.

\end{thebibliography}
\bibliographystyle{abbrv}

\end{document}